\documentclass[preprint,12pt]{elsarticle}

\usepackage{enumerate}
\usepackage{amssymb}
\usepackage{lipsum}
\usepackage[a4paper, total={6.2in, 9.2in}]{geometry}
\usepackage{booktabs}
\usepackage{amsmath,amssymb,url}
\usepackage{enumitem} 
\usepackage{graphics} 
\usepackage{array}
\usepackage{dsfont}
\usepackage[all]{xy}
\usepackage{amsthm}
\usepackage{tikz}

\usepackage{mathrsfs}

\numberwithin{equation}{section}
\newtheorem{theorem}{Theorem}[section]
\newtheorem{lemma}[theorem]{Lemma}

\newtheorem{corollary}[theorem]{Corollary}

\newtheorem{problem}[theorem]{Problem}

\newtheorem{definition}[theorem]{Definition}

\newtheorem{remark}[theorem]{Remark}
\newtheorem{example}[theorem]{Example}

\newcommand{\dn}{\mathord{\downarrow}\hspace{0.05em}}
\newcommand{\up}{\mathord{\uparrow}\hspace{0.05em}}

\newcommand\blfootnote[1]{%
\begingroup
\renewcommand\thefootnote{}\footnote{#1}%
\addtocounter{footnote}{-1}%
\endgroup
}

\journal{Topology and its Applications}

\begin{document}

\begin{frontmatter}



\title{Two problems of SI-compact spaces}


\author{Zhengmao He}
\address{School of Sciences, Southwest Petroleum University, Chengdu 610500, Sichuan, China}
\begin{abstract} Zhao and Ho use the irreducible open cover family with respected to the Scott topology to introduce the $SI$-compactness. It has been proved that the $SI$-compactness is strictly stronger than compactness. In this paper, we prove that the arbitrary product of  $SI$-compact spaces is always $SI$-compact which gives a positive answer for the question posed by Zhao and Ho. Furthermore, we show that the category of all $SI$-compact spaces is not a reflective full subcategory of $\mathbf{TOP}$.
\end{abstract}
\begin{keyword}  Scott topology; irreducible; $SI$-compact; reflective full subcategory\\
\vspace*{0.2cm}
\end{keyword}


\end{frontmatter}
\blfootnote{This work is supported by the Sichuan Science and Technology Program (Grant No. 2026NSFSC0784).}
\blfootnote{E-mail address: hezhengmaomath@163.com (Z.M.He).}


\section{Introduction}
\label{}

Compactness ranks among the most vital properties in general topology and has been widely investigated and thoroughly studied by topologists (see \cite{RE77}). Actually, a topological space $X$ is compact if and only if every directed open cover of $X$ contains $X$. In Domain theory, for a poset $P$, every directed set $D$ of $P$ is always irreducible with respect to the Scott topology. Irreducible sets are an important class of sets for investigating the sobriety of \(T_0\) spaces in Domain theory. In \cite{DME15}, Zhao and Ho introduced SI-compactness and essentially, they generalized directed open covers to Scott irreducible open covers. As we know, irreducible sets is strictly weaker than directed sets. For instance, for the dcpo $\mathcal{J}$ constructed by Johnstone, $\mathcal{J}$ is irreducible but not directed (see \cite{PJ81}). For more examples of irreducible sets that are not directed, we refer the reader to \cite{E2,ME182,TH62}. Accordingly, SI-compactness is stronger than compactness. In \cite{DME15}, Zhao and Ho raised the following two problems:

(1) Is there a compact space which is not $SI$-compact?

(2) Is the product of two $SI$-compact spaces always SI-compact?\\
In \cite{HWQ21}, the authors constructed a uncountable compact sober space which is not $SI$-compact by using the complete lattice constructed by Isbell (see \cite{IJ82}). Recently, it has proved that the countable Scott non-sober complete lattice constructed by Miao, Xi, Li and Zhao \cite{AB75} endowed with the lower topology is a countable compact and non $SI$-compact space (see \cite{A11}). Furthermore, it has proved that every topological space is a dense open subspace of a $SI$-compact space and the closed subspace of a $SI$-compact space is still $SI$-compact (see \cite{A11}). Thus, the first problem has a positive solution. In section 3, we show that the arbitrary product of $SI$-compact space is still $SI$-compact, thus the second question has also been resolved.

Reflective subcategories of the category $\mathbf{TOP}_{0}$ represent an important research focus in Domain theory. The categories of $d$-spaces, well-filtered spaces and sober spaces, three known classes of spaces, are all full reflective subcategories of $\mathbf{Top}_{0}$ (see \cite{GG03,I2,IOP82}).  For other reflective subcategories of the category $\mathbf{TOP}_{0}$, we refer the reader to \cite{I3,I4}.
Naturally, it arises the following question:

Is the category of all $SI$-compact spaces a reflective full subcategory of $\mathbf{Top}$?\\
In section 4, we give a negative answer for the above question. Specifically, we verify that the $SI$-compact reflection of the discrete space $\mathbb{N}$ does not exist.

\section{Preliminaries}
\label{}

\quad Let $P$ be a poset and $A\subseteq P$. We take $$\up A=\{x\in P\mid \exists\ a\in A, a\leq x \}$$ and $$\dn A=\{x\in P\mid \exists\ a\in A, x\leq a\}.$$
For every $x\in P$, we write $\dn x$ for $\dn\{x\}$  and $\up x$ for $\up\{x\}$. A nonempty subset $D\subseteq P$ is directed if $\forall\ a,b\in D$, $\up a\cap \up b \cap D\neq\emptyset$. Dually, a nonempty subset $D\subseteq P$ is filtered if $\forall\ a,b\in D$, $\dn a\cap \dn b \cap D\neq\emptyset$.

\quad A poset $P$ is called a {\em directed complete poset} ({\em  dcpo}, for short) if every
directed set of $P$ has a supremum.

\quad Let $P$ be a poset. A subset $U\subseteq P$ is {\em Scott open} (see \cite{GG03,JG13})if

 (i) $U=\up U$, and

 (ii) for every directed set $D$, $\bigvee D\in U$ implies
$D\cap U\neq \emptyset$, whenever $\bigvee D$ exists.\\
The all Scott open sets of $P$ form the Scott topology $\sigma(P)$. We write $(P,\sigma(P))$ as $\Sigma P$.

\quad Let $P,Q$ be two dcpos. Then the mapping $f:\Sigma P\longrightarrow \Sigma Q$ is continuous if and only if for every directed set $D\subseteq P$, $f(\bigvee D)=\bigvee\limits_{d\in D}f(d)$.

\quad Let $X$ be a topological space. A nonempty subset $F\subseteq X$ is {\em irreducible},
if for two closed sets $A$, $B\subseteq X$, $F\subseteq A\cup B$ implies $F\subseteq A$ or $F\subseteq B$. A $T_{0}$ space $X$ is {\em sober} if for every irreducible closed set $A$, there exists a $x\in X$ such that $A=\overline{\{x\}}$.

\quad Let $f:X\longrightarrow Y$ be a continuous map and $A$ an irreducible set of $X$. Then $f(A)$ is an irreducible set of $Y$.

\quad Given a $T_{0}$ space $X$, the {\em specialization order} $\leq$ on $X$ is given by
$$x\leq y \Longleftrightarrow x\in cl(\{y\}).$$
Unless otherwise stated, throughout
the paper, whenever an order-theoretic concept is mentioned in the context of a $T_{0}$ space $X$, it is to be interpreted with respect to the specialization order on $X$.

\quad Given a topological space $X$, the symbols $\mathcal{O}(X)$ represents the lattice of all open subsets of $X$ with inclusion order.

\quad In this paper, all topological spaces discussed are $T_{0}$.

\section{The product of SI-compact spaces}

\begin{definition} {\rm (see \cite{DME15})
A topological space $X$ is called $SI$-compact iff for every Scott irreducible subset $\mathcal{U}$ of $\mathcal{O}(X)$, $\bigcup\mathcal{U}=X$ implies $X\in\mathcal{U}$.}
\end{definition}

\begin{remark} {\rm(1) Every $SI$-compact space is always compact (see \cite{DME15}).

(2) There are a countable compact space and an uncountable compact space, both failing to be $SI$-compact(see \cite{HWQ21,A11}).

(3) Let $X$ be a topological space such that $\Sigma\mathcal{O}(X)$ is sober. Then $X$ is compact if and only if $X$ is $SI$-compact(see \cite{HWQ21})}
\end{remark}

\begin{lemma}{\rm Let $X,Y$ be a pair of topological spaces and $x\in X$. Then the mapping $$e_{x}:\Sigma\mathcal{O}(X\times Y)\longrightarrow\Sigma\mathcal{O}(Y)$$ is continuous, where $e_{x}(W)=\{y\in Y\mid (x,y)\in W\}$, for all $W\in\mathcal{O}(X\times Y)$.}
\end{lemma}

\begin{proof}Let $W$ be an open set of $X\times Y$. Then there are open sets families $\{U_{i}\mid i\in I\}\subseteq\mathcal{O}(X)$ and $\{V_{i}\mid i\in I\}\subseteq\mathcal{O}(Y)$ such that $W=\bigcup\limits_{i\in I}(U_{i}\times V_{i})$. So we have that $e_{x}(W)=\bigcup\mathscr{A}$ is open in $Y$, where $\mathscr{A}=\{V_{j}\mid j\in I, x\in U_{j}\}$. This implies that $e_{x}$ is well-defined. Let $\{W_{k}\mid k\in K\}\subseteq\mathcal{O}(X\times Y)$. One can check that
$$e_{x}(\bigcup\limits_{k\in K}W_{k})=\bigcup\limits_{k\in K}e_{x}(W_{k}).$$ Thus, $e_{x}$ is Scott continuous.\end{proof}
The following Lemma 3.4 is called the Tube Lemma(see \cite{RE77}). For the sake of completeness, we give a brief proof for the Lemma 3.4.

\begin{lemma}{\rm (see \cite{RE77}) Let $X,Y$ be a pair of topological spaces and $Y$ a compact space. If $W\in\mathcal{O}(X\times Y)$ and $\{x_{0}\}\times Y\subseteq W$ for some $x_{0}\in X$, then there is an open set $U\in\mathcal{O}(X)$ such that $\{x_{0}\}\times Y\subseteq U\times Y\subseteq W$.}
\end{lemma}

\begin{proof} Let $y\in Y$. Since $W\in\mathcal{O}(X\times Y)$, there are open sets $U_{y}\in\mathcal{O}(X)$ and $V_{y}\in\mathcal{O}(Y)$ such that $(x_{0},y)\in U_{y}\times V_{y}\subseteq W$. As $Y$ is compact and $Y=\bigcup_{y\in Y}V_{y}$, there exist $y_{1},y_{2},\cdot\cdot\cdot,y_{n}\in Y$ such that $Y=V_{y_{1}}\cup V_{y_{2}}\cup\cdot\cdot\cdot\cup V_{y_{n}}$. Set $U=U_{y_{1}}\cap U_{y_{2}}\cap\cdot\cdot\cdot\cap U_{y_{n}}$. Clearly, $x_{0}\in U\in\mathcal{O}(X)$ and further
$$U\times Y=(U_{y_{1}}\cap U_{y_{2}}\cap\cdot\cdot\cdot\cap U_{y_{n}})\cap(V_{y_{1}}\cup V_{y_{2}}\cup\cdot\cdot\cdot\cup V_{y_{n}})\subseteq \bigcup\limits_{1\leq k\leq n}(U_{y_{k}}\times V_{y_{k}})\subseteq W.$$ The proof is complete.\end{proof}

\begin{theorem}{\rm Let $X,Y$ be $SI$-compact spaces. Then $X\times Y$ is $SI$-compact.}
\end{theorem}
\begin{proof} Let $\mathscr{U}\subseteq\mathcal{O}(X\times Y)$ be a Scott irreducible open cover of $X\times Y$. Define a mapping $$\xi:\Sigma\mathcal{O}(X\times Y)\longrightarrow\Sigma\mathcal{O}(X)$$ as follows:
$$\forall\ W\in\mathcal{O}(X\times Y),\ \xi(W)=\{x\in X\mid \{x\}\times Y\subseteq W\}.$$

$\mathbf{Claim} \ 1$: $\xi$ is continuous.

Let $E\in\mathcal{O}(X\times Y)$ and $x\in \xi(E)$. Equivalently, $\{x\}\times Y\subseteq E$. Since $Y$ is $SI$-compact, $Y$ is compact. By Lemma 3.4, there is an open set $F\in\mathcal{O}(X)$ such that $\{x\}\times Y\subseteq F\times Y\subseteq E$. In other words, $x\in F\subseteq\xi(E)$. So $\xi(E)$ is open in $X$ and thus $\xi$ is well defined. Suppose $\{G_{i}\mid i\in I\}\subseteq\mathcal{O}(X\times Y)$ is a directed family. Obviously, $\bigcup\limits_{i\in I}\xi(G_{i})\subseteq \xi(\bigcup\limits_{i\in I}G_{i})$. Let $a\in\xi(\bigcup\limits_{i\in I}G_{i})$. Then by Lemma 3,3, $\{e_{a}(G_{i})\mid i\in I\}$ is an open cover of $Y$. By the compactness of $Y$, there are $G_{i_{1}},G_{i_{2}},\cdot\cdot\cdot,G_{i_{n}}$ such that $Y=e_{a}(G_{i_{1}})\cup e_{a}(G_{i_{2}})\cup\cdot\cdot\cdot\cup e_{a}(G_{i_{n}})$. This means that $\{a\}\times Y\subseteq G_{i_{1}}\cup G_{i_{2}}\cup\cdot\cdot\cdot\cup G_{i_{n}}$. Using the directedness of the family $\{G_{i}\mid i\in I\}$, we can choose a $G_{i^{\ast}}$ such that $G_{i_{1}},G_{i_{2}},\cdot\cdot\cdot,G_{i_{n}}\subseteq G_{i^{\ast}}$. Consequently, $\{a\}\times Y\subseteq G_{i^{\ast}}$ and hence $a\in\xi(G_{i^{\ast}})$. Thus, $\bigcup\limits_{i\in I}\xi(G_{i})=\xi(\bigcup\limits_{i\in I}G_{i})$. Therefore, $\xi$ is continuous.

$\mathbf{Claim}\ 2$: for every $x\in X$, there is a $U_{x}\in\mathscr{U}$ such that $x\in \xi(U_{x})$

By Lemma 3.3, $e_{x}(\mathscr{U})$ is irreducible in $\Sigma\mathcal{O}(X)$. Furthermore, $$\bigcup e_{x}(\mathscr{U})=\bigcup\limits_{U\in\mathscr{U}}e_{x}(U)=Y$$ because $\bigcup\mathcal{U}=X\times Y$. By the $SI$-compactness of $Y$, there is a $U_{x}\in\mathscr{U}$ satisfying $e_{x}(U_{x})=Y$. Whence, $\{x\}\times Y\subseteq U_{x}$. Thus, $x\in\xi(U_{x})$.

By Claim 2, $X\subseteq\bigcup\limits_{x\in X}\xi(U_{x})\subseteq\bigcup\xi(\mathscr{U})=\bigcup\limits_{U\in\mathscr{U}}\xi(U)$. So $\xi(\mathscr{U})$ is an open cover of $X$. By Claim 1, $\xi(\mathscr{U})$ is irreducible in $\Sigma\mathcal{O}(X)$. As $X$ is $SI$-compact, $X\in\xi(\mathscr{U})$. So there is a $U^{\ast}\in\mathscr{U}$ such that $X=\xi(U^{\ast})$. As a consequence, $X\times Y=U^{\ast}\in\mathscr{U}$. Therefore, $X\times Y$ is $SI$-compact.
\end{proof}

\begin{corollary}{\rm Let $X_{1},X_{2},\cdot\cdot\cdot,X_{n}$ be $SI$-compact spaces. Then $X_{1}\times X_{2}\times\cdot\cdot\cdot\times X_{n}$ is $SI$-compact.}
\end{corollary}

\begin{lemma}{\rm (see \cite{RE77}) Let $\{X_{i}\mid i\in I\}$ be a family of compact spaces. Then $\prod\limits_{i\in I}X_{i}$ is compact.}
\end{lemma}

\begin{theorem} {\rm Let $\{X_{i}\mid i\in I\}$ be a family of $SI$-compact spaces. Then $\prod\limits_{i\in I}X_{i}$ is $SI$-compact.}

\end{theorem} 

\begin{proof} Let $\mathscr{A}\subseteq\mathcal{O}(\prod\limits_{i\in I}X_{i})$ be an irreducible Scott closed subset such that $\bigcup\mathscr{A}=\prod\limits_{i\in I}X_{i}$. Note that $$\mathscr{D}=\{\prod\limits_{i\in I}U_{i}\mid \forall i\in I, U_{i}\in\mathcal{O}(X_{i})\text{ and }\{j\in I\mid U_{j}\neq X_{j}\text{ is finite}\}\}$$ forms a basis of $\prod\limits_{i\in I}X_{i}$. Set $$\mathscr{B}=\{W\in\mathscr{D}\mid \exists U\in\mathscr{A}, W\subseteq U\}.$$  It follows from $\bigcup\mathscr{A}=\prod\limits_{i\in I}X_{i}$ that $\mathscr{B}$ also covers $\bigcup\mathscr{A}=\prod\limits_{i\in I}X_{i}$. By Remark 3.2 and Lemma 3.7, there are $W_{1},\cdot\cdot\cdot,W_{n}\in\mathscr{B}$ such that $$W_{1}\cup\cdot\cdot\cdot\cup W_{n}=\prod\limits_{i\in I}X_{i}.$$ For each $1\leq k\leq n$, choose a $U_{k}\in\mathscr{A}$ with $W_{k}\subseteq U_{k}$. Since each $W_{k}$ belongs to $\mathscr{D}$, there exist a finite subset $F_{k}\subseteq I$ and $O_{k,i}\in O(X_{i})$ such that $$W_{k}=\bigcap\limits_{i\in S_{k}}\pi_{i}^{-1}(O_{k,i}),$$ where $\pi_{i}:\prod\limits_{i\in I}X_{i}\longrightarrow X_{i}$ is a projection. Take $$J=\bigcup\limits_{k=1}\limits^{n}S_{k}.$$
Then $J$ is finite. Fix $$X_{J}=\prod\limits_{j\in J}X_{j}\text{ and }Y=\prod\limits_{j\in I\setminus J}X_{j}.$$ One can check that $\prod\limits_{i\in I}X_{i}$ is homeomorphic to $X_{J}\times Y$. In the following, we regard $\prod\limits_{i\in I}X_{i}$ and $X_{J}\times Y$ as the same space up to homeomorphism. So we can write $W_{n}$ as $H_{k}\times Y$, where $H_{k}=\prod\limits_{i\in S_{k}}O_{k,i}\in\mathcal{O}(X_{J})$. We assert that $X_{J}=\bigcup\limits_{k=1}\limits^{n}H_{k}$. For each $j\in I\setminus J$, fix a $y_{j}\in X_{j}$ and write $y=(y_{j})_{j\in J}$. Choose any point $x=(x_{j})_{j\in J}\in X_{J}$. Then $(x,y)\in W_{1}\cup\cdot\cdot\cdot\cup W_{n}$. Select a $k$ with $(x,y)\in W_{k}=H_{k}\times Y$. This leads to $x\in H_{k}$. Consequently, $X_{J}\subseteq\bigcup\limits_{k=1}\limits^{n}H_{k}$. This proves $X_{J}=\bigcup\limits_{k=1}\limits^{n}H_{k}$. Define a mapping $\xi:\mathcal{O}(X_{J}\times Y)\longrightarrow\mathcal{O}(X_{J})$ by
$$\forall G\in\mathcal{O}(X_{J}\times Y),\ \xi(G)=\{x\in X_{J}\mid \{x\}\times Y\subseteq G\}.$$
By Remark 3.2 and Lemma 3.7 again, $Y$ is compact. Using the proof of Claim 1 in Theorem 3.5, $\xi$ is Scott continuous. As a result, $\xi(\mathscr{A})$ is still irreducible in $\Sigma\mathcal{O}(X_{J})$. Besides, $H_{k}\subseteq\xi(W_{k})$ because $H_{k}\times Y=B_{k}\subseteq W_{k}$. So we know $$X_{J}=\bigcup\limits_{k=1}\limits^{n}H_{k}\subseteq\bigcup\limits_{k=1}\limits^{n}\xi(W_{k})\subseteq\bigcup\xi(\mathscr{A}).$$
Consequently, $\xi(\mathscr{A})$ covers $X_{J}$. By Corollary 3.6 and the finiteness of $J$, $X_{J}$ is $SI$-compact. Applying Definition 3.1, we obtain $X_{J}\in\xi(\mathscr{A})$. Select $A\in\mathscr{A}$ with $X_{J}=\xi(A)$. This yields that $$X_{J}\times Y=\bigcup\limits_{x\in X_{J}}(\{x\}\times Y)\subseteq A.$$ In other words, $A=\prod\limits_{i\in I}X_{i}\in\mathscr{A}$. Therefore, $\prod\limits_{i\in I}X_{i}$ is $SI$-compact.
\end{proof}

In \cite{A11}, He and Zhao construct a countable compact and non $SI$-compact space by using the lower topology on dcpos. In addition, the product topology of the lower topologies on posets $P$ and $Q$ does not coincide with the lower topology on $P\times Q$. It arises the following problem.

\begin{problem}{\rm Let $\{P_{i}\mid i\in I\}$ be a family of posets such that each space $(P_{i},\omega(P_{i}))$ is $SI$-compact. Is $(\prod\limits_{i\in I}P_{i},\omega(\prod\limits_{i\in I}P_{i}))$ $SI$-compact?}
\end{problem}

\section{The category of $SI$-compact spaces is not reflective in $\mathbf{TOP}$}

\begin{definition} {\rm 
Let $X$ be a topological space. We say the pair $(\widehat{X},\eta)$ is the $SI$-compact reflection if $\widehat{X}$ is $SI$-compact and $\eta:X\longrightarrow \widehat{X}$ is continuous such that for every $SI$-compact space $Y$ and continuous mapping $f:X\longrightarrow Y$, there exists a unique continuous mapping $\overline{f}:\widehat{X}\longrightarrow Y$ such that $f=\overline{f}\circ\eta$.}
\end{definition}

\begin{center}
\centering
\includegraphics[totalheight=1.5in]{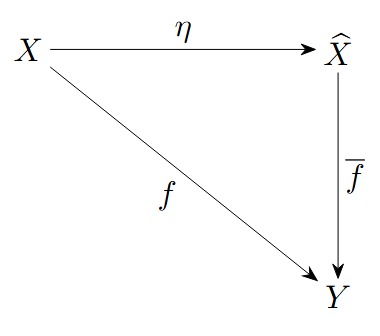}\\ The $SI$-compact reflection
\end{center}

Let $X, Y$ be topological spaces. The set of all continuous mappings from $X$ to $Y$ is denoted by $\mathrm{TOP}(X,Y)$. The order on $\mathrm{TOP}(X,Y)$ is defined by
 $$\forall\ f,g\in\mathrm{TOP}(X,Y), f\preceq g \mbox{ if and only if}\ \forall x\in X, f(x)\leq g(x)\ \mbox{in}\ Y.$$

\begin{example} {\rm Equip the natural numbers set $\mathbb{N}$ with the discrete topology. Assume that $(SI(\mathbb{N}),\eta)$ is the $SI$-compact reflection of $\mathbb{N}$. Then for every $SI$-compact space $Y$, there is an one-to-one correspondence between the set of all continuous mappings $\mathrm{TOP}(\mathbb{N},Y)$ and the set of all continuous mappings $\mathrm{TOP}(SI(\mathbb{N},Y))$. Actually, this one-to-one correspondence $G_{Y}$ sends every $f\in\mathrm{TOP}(SI(\mathbb{N}),Y))$ to $f\circ\eta\in\mathrm{TOP}(\mathbb{N},Y)$. Take $Y=\mathbb{S}$, where $\mathbb{S}$ is the Sierpi\'{n}ski space. Specifically,
$$\mathbb{S}=\{0,1\}\ \mbox{and}\ \mathcal{O}(\mathbb{S})=\{\emptyset,\{1\},\{0,1\}\}.$$
Obviously,\ $\mathbb{S}$ is $SI$-compact. Given a topological space $T$, $\mathrm{TOP}(T,\mathbb{S})$ is order isomorphic to $\mathcal{O}(T)$. The order homeomorphism $F_{T}$ is given by sending every $f\in\mathrm{TOP}(T,\mathbb{S})$ to $f^{-1}(\{1\})$. The inverse mapping $F^{-1}_{T}$ sends $U\in\mathcal{O}(T)$ to $f_{U}$, where $f_{U}: T\longrightarrow\mathbb{S}$ is defined by
 $$\forall \ t\in T, \ f_{U}(t)=\left\{
             \begin{array}{ll}
              1, &\ \ t\in U, \\
              0, &\ \ t\in T\setminus U.
             \end{array}
           \right.$$

 Hence, $\mathrm{TOP}(\mathbb{N},\mathbb{S})$ and $\mathrm{TOP}(SI(\mathbb{N}),\mathbb{S})$ are order isomorphic to $\mathcal{O}(\mathbb{N})$ and $\mathcal{O}(SI(\mathbb{N}))$ respectively.
\begin{center}
\centering
\includegraphics[totalheight=1.2in]{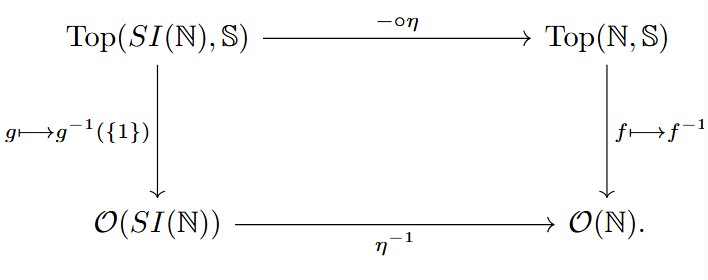}\\ The mapping between $\mathcal{O}(\mathbb{N})$ and $\mathcal{O}(SI(\mathbb{N}))$
\end{center}
Now, we can conclude that the mapping $F_{\mathbb{N}}\circ G_{\mathbb{S}}\circ F^{-1}_{SI(\mathbb{N})}$ is an order-preserving bijection. $\forall\ W\in\mathcal{O}(SI(\mathbb{N}))$,
\begin{center}
$\begin{array}{lll}
F_{\mathbb{N}}\circ G_{\mathbb{S}}\circ F^{-1}_{SI(\mathbb{N})}(W)&=&F_{\mathbb{N}}\circ G_{\mathbb{S}}(f_{W})\\
&=&F_{\mathbb{N}}(f_{W}\circ\eta)\\
&=&(f_{W}\circ\eta)^{-1}(\{1\})\\
&=&\{n\in\mathbb{N}\mid f_{W}(\eta(n))=1\}\\
&=&\{n\in\mathbb{N}\mid \eta(n)\in W\}\\
&=&\eta^{-1}(W).\\
\end{array}
$\end{center}
Consequently, $\eta^{-1}:\mathcal{O}(SI(\mathbb{N}))\longrightarrow\mathcal{O}(\mathbb{N})$ is a bijection. Note that the mapping $\eta^{-1}$ is the preimage mapping of $\eta$ not the inverse of the mapping $\eta$. The inverse of the mapping $\eta^{-1}$ is denoted by $$\alpha:\mathcal{O}(\mathbb{N})\longrightarrow\mathcal{O}(SI(\mathbb{N})).$$ Then for every $A\in\mathcal{O}(\mathbb{N})$, $\eta^{-1}(\alpha(A))=A$. Since $\eta^{-1}$ is an order-preserving bijection, $\eta^{-1}(\emptyset)=\emptyset$ and $\eta^{-1}(SI(\mathbb{N}))=\mathbb{N}$. So we have $\alpha(\emptyset)=\emptyset$ and $\alpha(\mathbb{N})=SI(\mathbb{N})$. Furthermore, $\forall\ B,C\in\mathcal{O}(\mathbb{N}), \{A_{i}\mid i\in I\}\subseteq\mathcal{O}(\mathbb{N})$,
$$\eta^{-1}(\alpha(A)\cap\alpha(B))=\eta^{-1}(\alpha(A))\cap\eta^{-1}(\alpha(B))=A\cap B=\eta^{-1}(\alpha(A\cap B)),$$ and
$$\eta^{-1}(\bigcup\limits_{i\in I}\alpha(A_{i}))=\bigcup\limits_{i\in I}\eta^{-1}(\alpha(A_{i}))=\bigcup\limits_{i\in I}A_{i}=\eta^{-1}(\alpha(\bigcup\limits_{i\in I}A_{i})).$$ As $\eta^{-1}$ is a bijection, $$\alpha(A)\cap\alpha(B)=\alpha(A\cap B)\ \mbox{and}\ \bigcup\limits_{i\in I}\alpha(A_{i})=\alpha(\bigcup\limits_{i\in I}A_{i}).$$

$\mathbf{Claim} \ 1$: $\eta$ is a injection.

$\forall\ m,n\in\mathbb{\mathbb{N}}$ and $m\neq n$, $\eta^{-1}(\alpha(\{n\}))=\{n\}$. Then $\eta(n)\in\alpha(\{n\})$ and $\eta(m)\not\in\alpha(\{n\})$. So $\eta(n)\neq\eta(m)$.

$\mathbf{Claim} \ 2$: $\forall\ x\in SI(\mathbb{N})$, there exists a unique $n_{x}\in\mathbb{N}$ such that $x\in\alpha(\{n_{x}\})$.

Note that $SI(\mathbb{N})=\alpha(\mathbb{N})=\bigcup\limits_{p\in\mathbb{N}}\alpha(\{p\})$. Then there is a $n_{x}\in\mathbb{N}$ such that $x\in\alpha(\{n_{x}\})$. As $\eta^{-1}$ is a bijection, $n_{x}$ is unique.

Define a mapping $h:SI(\mathbb{N})\longrightarrow SI(\mathbb{N})$ by $h(x)=\eta(n_{x})$.

$\mathbf{Claim} \ 3$: $h$ is continuous and $h\circ\eta=\eta$.

$\forall\ D\in\mathcal{O}(SI(\mathbb{N}))$, there is a $E\in\mathcal{O}(\mathbb{N})$ such that $\alpha(E)=D$. $\forall\ a\in\alpha(E)$,
$$p\in\alpha(E)\Longleftrightarrow p\in\bigcup\limits_{e\in E}\alpha(\{e\})\Longleftrightarrow n_{p}\in E$$ and
$$h(p)\in\alpha(E)\Longleftrightarrow\eta(n_{p})\in\alpha(E)\Longleftrightarrow n_{p}\in\eta^{-1}(\alpha(E))=E.$$ Consequently, $h^{-1}(D)=h^{-1}(\alpha(E))=\alpha(E)\in\mathcal{O}(SI(\mathbb{N}))$. Thus, $h$ is continuous.

$\forall\ s\in\mathbb{N}$, $\eta(s)\in\alpha(\{s\})$. By the uniqueness of $n_{\eta(s)}$, $s=n_{\eta(s)}$. Now, $$h\circ\eta(s)=h(\eta(s)=\eta(n_{\eta(s)})=\eta(s).$$ As a consequence, $h\circ\eta=\eta$.

$\mathbf{Claim} \ 4$: $h=id_{SI(\mathbb{N})}$.

As $SI(\mathbb{N})$ is the $SI$-compact reflection and $h\circ\eta=id_{SI(\mathbb{N})}\circ\eta=\eta$, $h=id_{SI(\mathbb{N})}$.

By Claim 4, $\forall\ y\in SI(\mathbb{N})$, $h(y)=\eta(n_{y})=y$. So $\eta$ is a surjection and hence $SI(\mathbb{N})=\eta(\mathbb{N})$.

$\mathbf{Claim} \ 5$: $SI(\mathbb{N})$ is a discrete space.

Firstly, we show that for all $K\subseteq\mathbb{N}$, $\alpha(K))=\eta(K)$. $\forall\ k\in K$, $k\in K=\eta^{-1}(\alpha(K))$. So $\eta(k)\in\alpha(K)$ and hence $\eta(K)\subseteq\alpha(K))$.
On the contrary, $\forall\ c\in\alpha(K)$, there is a unique $n_{c}\in K$ such that $c\in\alpha(\{n_{c}\})$. Besides, by Claim 4, $c=h(c)=\eta(n_{c})\in\eta(K)$. Hence, $\alpha(K))=\eta(K)$. $\forall z\in SI(\mathbb{N})$, there is a $n_{z}\in\mathbb{N}$ such that $z=h(z)=\eta(n_{z})$. Now, $\{z\}=\eta(\{n_{z}\})=\alpha(\{n_{z}\})\in\mathcal{O}(SI(\mathbb{N}))$. In other words, $\{z\}$ is open in $SI(\mathbb{N})$. Therefore, $SI(\mathbb{N})$ is a discrete space.

As $\eta$ is a injection, $SI(\mathbb{N})$ is finite. A infinite discrete space is not compact and surely non $SI$-compact. This contradicts that $SI(\mathbb{N})$ is the $SI$-compact reflection of $\mathbb{N}$.}
\end{example}
By Example 4.2, we have the following Theorem 4.3 immediately.

\begin{theorem} The category of all $SI$-compact spaces is not reflective in the category of $\mathbf{TOP}$.
\end{theorem}

\end{document}